\documentclass[11pt]{article}
\usepackage[cp1251]{inputenc}
\usepackage{amsthm,amssymb,amsmath}
\usepackage{color}
\usepackage{amsthm,amssymb,amsmath,latexsym}
\usepackage{graphicx}

\theoremstyle{definition}
\newtheorem{defin}{Definition}[section]
\theoremstyle{remark}
\newtheorem{example}[defin]{Example}

\newtheorem{remar}[defin]{Remark}
\theoremstyle{plain}
\newtheorem{thm}[defin]{Theorem}
\newtheorem{prop}[defin]{Proposition}

\newtheorem{corol}[defin]{Corollary}

\newcommand{\wh}{\widehat}
\newcommand{\ol}{\overline}

\numberwithin{equation}{section}

\begin{document}

\title{Subdiagrams of Bratteli diagrams supporting finite invariant measures}

\date{}

\author{{\bf S. Bezuglyi}, {\bf O. Karpel} and {\bf J. Kwiatkowski}}

\maketitle

\begin{abstract}
We study finite measures on Bratteli diagrams invariant with respect to the tail equivalence relation. Amongst the proved results on finiteness of measure extension, we characterize the vertices of a Bratteli diagram that support an ergodic finite invariant measure.

\end{abstract}

\section{Introduction and background}

In this note, we continue the study of ergodic measures on the path space $X_B$ of a Bratteli diagram $B$ started in \cite{BKMS_2010} and \cite{BKMS_FinRank}. Recall that, given a minimal (or even aperiodic) homeomorphism $T$ of a Cantor set $X$, one can construct a refining sequence $(\xi_n)$ (beginning with $\xi_0 = X$) of clopen partitions such that every $\xi_n$ is a finite collection of $T$-towers $(X_v^{(n)} : v \in V_n)$ \cite{herman_putnam_skau:1992}, \cite{giordano_putnam_skau:1995}, \cite{medynets:2006}. This fact is in the base of the very fruitful idea:  $(X,T)$ can be realized as a homeomorphism $\varphi$ (Vershik map) acting on the path space of a Bratteli diagram. By definition, a Bratteli diagram $B$ is represented as an infinite graph with the set of vertices $V$ partitioned into levels $V_n,\ n\geq 0,$ such that the edge set $E_n$ between levels $n-1$ and $n$ is determined by the intersection of towers of partitions $\xi_{n-1}$ and $\xi_n$ (the detailed definition and references are given below). Every $T$-invariant (hence, $\varphi$-invariant) measure $\mu$ on $X$  is completely defined by its values $\mu(X_v^{(n)})$ on all towers where $v \in V_n$ and  $n\geq 0$. In \cite{BKMS_2010} and  \cite{BKMS_FinRank}, the cases of stationary and finite rank Bratteli diagrams (i.e. $|V_n| \leq d$  for all $n$) were studied. We notice that, while studying $\varphi$-invariant measures, we can ignore some rather subtle questions about the existence of a Vershik map on the path space (see \cite{BKY}, \cite{BY}) and work with the measures invariant with respect to the tail equivalence relation $\mathcal E$ (cofinal equivalence relation, in other words).

Our  interest and motivation for this work arises from the following result proved in \cite{BKMS_FinRank}:   for any ergodic probability measure $\mu$ on a finite rank diagram $B$, there exists a subdiagram $\ol B$ of $B$ defined by a sequence $W= (W_n)$, where $W_n \subset V_n$, such that $\mu(X_w^{(n)})$ is bounded from zero for all $w \in W_n$ and $n$.
It was also shown that $\mu$ can be obtained as an extension of an ergodic measure on the subdiagram $\overline B$, in other words, $\overline B$ supports $\mu$ (the detailed definitions can be found below).



What is  an analogue of the above result for general Bratteli diagrams? Suppose we take a subdiagram $\ol B = \ol B(W)$ of a Bratteli diagram $B$ and consider an ergodic probability measure $\nu$ on $\ol B$. Then this measure can be naturally extended (by $\mathcal E$-invariance) to a measure $\wh\nu$ defined on the $\mathcal E$-saturation $\wh X_{\ol B}$ of the path space $X_{\ol B}$.
When the cardinality of $W_n$ is growing, then we cannot expect that the measures of the towers corresponding to the vertices from $W_n$ are bounded from below.
But we do expect that the rate of changes of $\widehat{\nu}(X_v^{(n)})$ is essentially different for $v\in W_n$ and $v\notin W_n$.
We prove that if the measure $\wh\nu$ is finite and the ratio $\frac{|W_{n}|}{|V \setminus W_{n}|}$ is bounded, then the minimal value of $\{\wh\nu (X^{(n)}_v) : v \notin W_n\}$ is much smaller than the maximal value $\{\wh\nu (X^{(n)}_v) : v \in W_n\}$. We also get the results for the ratio of the tower heights corresponding to $W_n$ and $V \setminus W_{n}$.


Another assertion that is proved in the paper is a modification of \cite[Theorem 6.1]{BKMS_FinRank}. We also give a criterion for the finiteness of the extended measure $\wh\nu$, using the condition on entries of the incidence matrices. A number of examples related to this issue is also considered in the paper.

The most of definitions and notation used in this paper are taken from  \cite{BKMS_FinRank}. Since the concept of Bratteli diagrams has been studied in a great number of recent research papers devoted to various aspects of Cantor dynamics, we give here only some  necessary definitions and notation referring to the pioneering articles \cite{herman_putnam_skau:1992}, \cite{giordano_putnam_skau:1995} (see also  \cite{berthe:2010}, \cite{BKMS_FinRank}) where the reader can find more detailed definitions and the widely used techniques, for instance, the  telescoping procedure.

A {\it Bratteli diagram} is an infinite graph $B=(V,E)$ such that the vertex
set $V =\bigcup_{i\geq 0}V_i$ and the edge set $E=\bigcup_{i\geq 1}E_i$
are partitioned into disjoint subsets $V_i$ and $E_i$ where

(i) $V_0=\{v_0\}$ is a single point;

(ii) $V_i$ and $E_i$ are finite sets;

(iii) there exists a range map $r$ and a source map $s$, both from $E$ to
$V$, such that $r(E_i)= V_i$, $s(E_i)= V_{i-1}$, and
$s^{-1}(v)\neq\emptyset$, $r^{-1}(v')\neq\emptyset$ for all $v\in V$
and $v'\in V\setminus V_0$.

Given a Bratteli diagram $B$, the $n$-th {\em incidence matrix}
$F_{n}=(f^{(n)}_{v,w}),\ n\geq 0,$ is a $|V_{n+1}|\times |V_n|$
matrix such that $f^{(n)}_{v,w} = |\{e\in E_{n+1} : r(e) = v, s(e) = w\}|$
for  $v\in V_{n+1}$ and $w\in V_{n}$. Here the symbol $|\cdot |$ denotes the cardinality of a set.

For a Bratteli diagram $B = (V,E)$, the set of all infinite paths in $B$ is denoted by $X_B$.  The topology defined by finite paths (cylinder sets) turns $X_B$ into a 0-dimensional metric compact space. We will consider only such Bratteli diagrams for which $X_B$ is a {\em Cantor set}.  The tail equivalence relation $\mathcal E$ on $X_B$  says that two paths $x =(x_n)$ and $y =(y_n)$ are tail equivalent  if and only if $x_n = y_n$ for $n$ sufficiently large. Let $\overline W = \{W_n\}_{n>0}$ be a sequence of (proper, non-empty) subsets $W_n$ of $V_n$. Set $W'_n = V_n \setminus W_n$. The (vertex) subdiagram $\overline B =  (\ol W, \ol E)$ is defined by the vertices $\ol W = \bigcup_{i\geq 0} W_n$ and the edges $\ol E$ that have their source and range in $\ol W$. In other words,
the incidence matrix $\ol F_n$ of $\ol B$ is defined by those edges from $B$ that have their source and range in vertices from $W_{n}$ and $W_{n+1}$, respectively.

We use also the following notation for an $\mathcal E$-invariant measure $\mu$ on $X_B$ and $n \geq 1$ and $v \in V_n$:

\begin{itemize}
\item  $X_v^{(n)}\subset X_B$ denotes the set of all paths that go through the vertex $v$;

\item  $h_v^{(n)}$ denotes the cardinality of the set of all finite paths (cylinder sets) between $v_0$ and $v$;

\item $p_v^{(n)}$ denotes the $\mu$-measure of the cylinder set $e(v_0,v)$ corresponding to a finite path between $v_0$ and $v$ (since $\mu$ is $\mathcal E$-invariant, the value $p_v^{(n)}$ does not depend on $e(v_0,v)$).
\end{itemize}

If $\ol B$ is a subdiagram defined by a sequence $\ol W = (W_n)$, then we use the notation $\ol X_w^{(n)}$ and $\ol h_w^{(n)}$ to denote the corresponding objects of the subdiagram $\ol B$.

Take a subdiagram $\ol B$ and consider the set $X_{\ol B}$ of all infinite paths whose edges belong to $\ol B$. Let $\widehat X_{\ol B} := \mathcal E(X_{\ol B})$ be the subset of paths in $X_B$ that are tail equivalent to paths from $X_{\ol B}$. In other words, the $\mathcal E$-invariant subset $\widehat X_{\ol B} $ of $X_B$ is the saturation of $X_{\ol B}$ with respect to the equivalence relation $\mathcal E$ (or $X_{\ol B}$ is a countable complete section of  $\mathcal E$ on $\widehat X_{\ol B}$). Let $\mu$ be a probability measure on $X_{\ol B}$ invariant with respect to the tail equivalence relation defined on $\ol B$. Then $\mu$ can be canonically extended to the measure $\widehat \mu$ on the space $\widehat X_{\ol B}$ by invariance with respect to $\mathcal E$ \cite{BKMS_FinRank}. If we want to extend $\widehat{\mu}$ to the whole space $X_{B}$, we can set $\widehat {\mu} (X_B \setminus \widehat{X}_{\ol B}) = 0$.

Specifically, take a finite path $\ol e \in \ol E(v_0, w)$ from the top vertex $v_0$ to a vertex $w \in W_n$  that belongs to the subdiagram $\ol B$.  Let $[\ol e]$ denote the cylinder subset of $X_{\ol B}$ determined by $\ol e$. For any finite path $s \in E(v_0, w)$ from the diagram $B$ with the same range $w$, we set $\widehat\mu ([s])  = \mu([\ol e])$. In such a way, the measure $\widehat\mu$ is defined on the $\sigma$-algebra of Borel subsets of $\wh X_{\ol B}$ generated by all clopen sets of the form $[z]$ where a finite path $z$ has the range in a vertex from $\ol B$. Clearly, the restriction of $\wh\mu$ on $X_{\ol B}$ coincides with $\mu$. We note that  the value $\wh\mu(\wh X_{\ol B})$ can be either finite or infinite depending on the structure of $\ol B$ and $B$ (see below Theorems \ref{neccsuff} and \ref{criterion}). Furthermore, the {\it support} of $\widehat\mu$ is, by definition, the set $\wh X_{\ol B}$. Set
\begin{equation}\label{n-th level}
\wh X_{\ol B}^{(n)} = \{x = (x_i)\in \wh X_{\ol B} : r(x_i) \in W_i, \ \forall i \geq n\}.
\end{equation}
Then $\wh X_{\ol B}^{(n)} \subset \wh X_{\ol B}^{(n+1)}$ and

\begin{equation}\label{extension_method}
\widehat\mu(\wh X_{\ol B}) = \lim_{n\to\infty} \widehat\mu(\wh X_{\ol B}^{(n)})
= \lim_{n\to\infty}\sum_{w\in W_n}  h^{(n)}_w p^{(n)}_w.
\end{equation}

\section{Characterization of subdiagrams supporting a measure}

Given a Bratteli diagram $B$, we consider the incidence matrix  $F_n = (f^{(n)}_{v,w}), v\in V_{n+1}, w\in V_{n}$ and set $A_n = F_n^T$, the transpose of $F_n$.
Together with the sequence of incidence matrices $(F_n)$, we consider the sequence of stochastic matrices $(Q_n)$ whose entries are:
$$
q^{(n)}_{v,w} = f^{(n)}_{v,w}\frac{h_w^{(n)}}{h_v^{(n+1)}},\ v\in V_{n+1}, w\in V_{n}.
$$

The following result was  obtained in \cite[Proposition 6.1]{BKMS_FinRank}  for  Bratteli diagrams of finite rank. We note here that {\em this result remains true for arbitrary Bratteli diagrams}, the proof is the same as in \cite{BKMS_FinRank}.

\begin{thm}\label{neccsuff}
Let $B$ be a Bratteli diagram with incidence stochastic matrices $\{Q_n = (q_{v,w}^{(n)})\}$ and let $\ol B$ be a proper vertex subdiagram of $B$ defined by a sequence of subsets $(W_n)$ where $W_n \subset V_n$.

(1) Let $\mu$ be a probability invariant measure on the path space $X_{\ol B}$ such that the extension $\widehat{\mu}$ of $\mu$ on $\wh X_{\ol B}$ is finite. Then

\begin{equation}\label{neces2}
\sum_{n=1}^{\infty} \sum_{w \in W_{n+1}} \sum_{v \in W^{'}_{n}} q_{w,v}^{(n)} \mu(\ol X_{w}^{(n+1)})< \infty.
\end{equation}

(2) If
\begin{equation}\label{suffic}
\sum_{n=1}^{\infty} \sum_{w \in W_{n+1}} \sum_{v \in W^{'}_{n}} q_{w,v}^{(n)} < \infty,
\end{equation}
then any probability invariant measure $\mu$ defined on the path space $X_{\ol B}$ of the subdiagram $\ol B$ extends to a finite measure $\widehat{\mu}$ on $\wh X_{\ol B}$.
\end{thm}

The following example shows that in general case, sufficient condition (\ref{suffic}) is not necessary and necessary condition  (\ref{neces2}) is not sufficient.

\begin{example}\label{example}
(1) First, we give an example of an infinite measure $\widehat{\mu}$ on a Bratteli diagram $B$ such that $\widehat{\mu}$ is an extension of a probability measure $\mu$ from a subdiagram $\overline{B}(\overline{W})$ and condition (\ref{neces2}) is satisfied.

Let $B$ be a stationary Bratteli diagram with incidence matrix
$$
F =
\begin{pmatrix}
3 & 0 & 0\\
1 & 2 & 0\\
0 & 1 & 3
\end{pmatrix}.
$$
Suppose the sequence $(W_n)$ is stationary and formed by the second and third vertices of each level. Then $(W'_n)$ is formed by the first vertex.
Since $q_{3,1} = 0$, we have
$$
\sum_{n=1}^{\infty} \sum_{v \in W_{n+1}} \sum_{w \in W^{'}_{n}} q_{v,w}^{(n)} \mu(\ol X_{v}^{(n+1)}) = \sum_{n=1}^{\infty}q_{2,1} \mu(\ol X_{2}^{(n+1)}).
$$
Compute
$$
q_{2,1} = \frac{h_{1}^{(n)}}{h_{2}^{(n+1)}} = \frac{3^{n-1}}{2^n + \sum_{k=0}^{n-1}2^k3^{n-1-k}} = \frac{3^{n-1}}{2^n + (3^n - 2^n)} = \frac{1}{3}.
$$
It is easy to see that
$$
\mu(\ol X_{2}^{(n+1)}) = \frac{2^{n-1}}{3^n}.
$$
Then
$$
q_{2,1} \mu(\ol X_{2}^{(n+1)}) = \frac{2^{n-1}}{3^{n+1}},
$$
hence condition (\ref{neces2}) is satisfied. On the other hand, we know that the extension $\wh\mu$ is an infinite measure because the Perron-Frobenius eigenvalue of the incidence matrix of $\ol B$ is 3, the same as for the odometer corresponding to the first vertex (see \cite{BKMS_2010}).

(2) For any stationary Bratteli diagram, sufficient condition (\ref{suffic}) is never satisfied. Thus, to show that  (\ref{suffic}) is not necessary, we can consider any stationary diagram with finite full measure $\wh\mu$. For instance, one can take the diagram with incidence matrix
$$
F =
\begin{pmatrix}
2 & 0 & 0\\
1 & 2 & 0\\
0 & 1 & 3
\end{pmatrix}
$$
and $\mu$ is the measure on the subdiagram $\ol B$ defined as in (1).
\end{example}

In contrast to Theorem \ref{neccsuff}, the following result gives a necessary and sufficient condition for finiteness of a measure extension.

\begin{thm}\label{criterion}
Let $B, \ol B, Q_n, W_n$ be as in Theorem \ref{neccsuff} and
$\mu$ a probability measure on the path space of the vertex  subdiagram $\ol B$.
The measure extension $\wh \mu(\wh X_{\ol B})$ is finite if and only if
\begin{equation}\label{neces}
\sum_{n=1}^{\infty} \sum_{w \in W_{n+1}} \wh\mu(X_{w}^{(n+1)})\sum_{v \in W^{'}_{n}} q_{w,v}^{(n)}
< \infty
\end{equation}
or
\begin{equation}\label{neces-1}
\sum_{i=1}^{\infty} \left(\sum_{w\in W_{i+1}} h_w^{(i+1)}p_w^{(i+1)} - \sum_{w\in W_{i}} h_w^{(i)}p_w^{(i)}\right) < \infty.
\end{equation}
\end{thm}

\begin{proof}  Indeed, let
$\wh X_{\ol B}^{(n)}$ be defined as in (\ref{n-th level}). Then $\wh\mu(\wh X_{\ol B}) = \lim_{n\to\infty}\wh\mu(\wh X_{\ol B}^{(n)})$. Since
$$
\wh X_{\ol B}^{(n)}  = \wh X_{\ol B}^{(1)} \cup (\wh X_{\ol B}^{(2)}\setminus \wh X_{\ol B}^{(1)}) \cup \cdots \cup (\wh X_{\ol B}^{(n)}\setminus \wh X_{\ol B}^{(n-1)}),
$$
we obtain
$$
\wh\mu (\wh X_{\ol B}^{(n)}) = 1 + \sum_{i=1}^{n-1}\left(\sum_{w\in W_{i+1}} h_w^{(i+1)}p_w^{(i+1)} - \sum_{w\in W_{i}} h_w^{(i)}p_w^{(i)}\right).
$$
This relation proves (\ref{neces-1}). We remark that condition (\ref{neces-1}) is formulated using the vertices related only to the subdiagram $\ol B$.

On the other hand,
$$
\wh X_{\ol B}^{(n)}\setminus \wh X_{\ol B}^{(n-1)} = \{ x = (x_i) \in \wh X_{\ol B} :
r(x_n) \notin W'_n,\ r(x_i) \in W_i,\ i\geq n+1\}.
$$
and therefore
\begin{eqnarray*}
\wh \mu(\wh X_{\ol B}^{(n)}\setminus \wh X_{\ol B}^{(n-1)}) & = &
\sum_{w\in W_{n+1}} \sum_{v\in W^{'}_n} f_{w,v}^{(n)}h_v^{(n)} p_w^{(n+1)}
\\
& = & \sum_{w\in W_{n+1}} \sum_{v\in W^{'}_n} q_{w,v}^{(n)}h_w^{(n+1)} p_w^{(n+1)} \\
&=& \sum_{w\in W_{n+1}} \wh\mu(X_w^{(n+1)}) \sum_{v\in W^{'}_n} q_{w,v}^{(n)}.
\end{eqnarray*}
Thus,
$$
 \wh\mu(\wh X_{\ol B}) = 1 + \sum_{n=1}^\infty \sum_{w\in W_{n+1}} \wh\mu(X_w^{(n+1)}) \sum_{v\in W^{'}_n} q_{w,v}^{(n)}.
$$

\end{proof}

To simplify the formulation of the next statement, we assume that $f_{w,v} > 0$ for every $w \in W_{n+1}$, $v \in W'_n$ and $n > 0$, i.e. for every $w \in W_{n+1}$ there is an edge to some vertex from $W'_n$. This assumption is not restrictive since one can use the telescoping procedure to ensure the positivity of all entries of $F$.

\begin{corol}\label{ratioheights}
Let $B, \ol B, Q_n, W_n$ be as in Theorem \ref{neccsuff} and $\mu$ a probability measure on the path space of the vertex  subdiagram $\ol B$.
Let the measure extension $\wh \mu(\wh X_{\ol B})$ be finite.
Then
$$
\sum_{n=1}^\infty \min_{w \in W_{n+1}}\max_{v \in W'_n} q_{w,v}^{(n)} <
\infty.
$$
In particular,
\begin{equation}\label{heights ratio}
\sum_{n=1}^{\infty}\min_{w \in W_{n+1}}\max_{v \in W'_n} \frac{h^{(n)}_v}{h^{(n+1)}_w} < \infty.
\end{equation}
\end{corol}

\begin{proof}
By Theorem~\ref{criterion}, we have
\begin{eqnarray*}
\wh\mu(\wh X_{\ol B}) & = & 1 + \sum_{n=1}^\infty \sum_{w\in W_{n+1}} \wh\mu(X_w^{(n+1)}) \sum_{v\in W^{'}_n} q_{w,v}^{(n)}\\
&\geq& 1+ \sum_{n=1}^\infty \sum_{w\in W_{n+1}} \wh\mu(X_w^{(n+1)}) \max_{v\in W^{'}_n}q_{w,v}^{(n)}\\
&\geq& 1+ \sum_{n=1}^\infty \min_{w \in W_{n+1}}\max_{v \in W'_n} q_{w,v}^{(n)} \sum_{w\in W_{n+1}} \wh\mu(X_w^{(n+1)}).
\end{eqnarray*}
Since
$$\sum_{w\in W_{n+1}} \wh\mu(X_w^{(n+1)}) \rightarrow
\wh\mu(\wh X_{\ol B}) > 0,
$$
there is a constant $C > 0$ such that $\sum_{w\in W_{n+1}}
\wh\mu(X_w^{(n+1)}) > C$ for all $n$. Hence we obtain
$$
\sum_{n=1}^\infty \min_{w \in W_{n+1}}\max_{v \in W'_n} q_{w,v}^{(n)} <
\infty.
$$
 Since $f_{w,v} > 0$ for every $w \in W_{n+1}$, $v \in W'_n$ and $n > 0$,
relation (\ref{heights ratio}) follows.
\end{proof}

\begin{remar}
Let $B$ be a stationary Bratteli diagram. If $B$ is simple then there is a unique ergodic invariant measure $\nu$ on $X_B$. Suppose that $\lambda$ is the  Perron-Frobenius eigenvalue for the incidence matrix of $B$. Then all the heights $h_v^{(n)}$ of $B$ grow as $\lambda^n$ and there is no proper subdiagram $\ol B$ such that $\nu$ could be the extension of an invariant ergodic measure from $X_{\ol B}$. In the case of a non-simple stationary diagram $B$, the minimal support of an ergodic invariant measure is some simple stationary subdiagram $\ol B(W)$ whose incidence matrix $\ol F$ has the Perron-Frobenius eigenvalue $\ol\lambda$. Then for every $w \in W_n$, $h_w^{(n)}$ grows again as $\ol\lambda^n$ but for every $v \in W_n'$, $h_v^{(n)}$ grows as $\delta^n$ where $\delta < \ol \lambda$ (see~\cite{BKMS_2010}).
\end{remar}


\medskip
We recall that, for a finite rank Bratteli diagram, the support of any probability measure $\mu$ is determined by a vertex subdiagram $\ol B(W), W = (W_n),$ whose vertices $v$ satisfy the condition: there exists some $\delta >0$ such that $\mu(X^{(n)}_v) > \delta$ for all sufficiently large $n$ and all $v\in W_n$ \cite{BKMS_FinRank}. In particular, a Bratteli diagram $B$ is of exact finite rank if  the condition  $\mu(X^{(n)}_v) > \delta$ holds for all vertices $v\in V_n$. Clearly, the above result cannot be true for general Bratteli diagrams. Nevertheless, we can find another characterization for vertices  that belong to the support of a probability measure by studying how the measure of towers $X^{(n)}_v$ changes when $v$ is in the subdiagram and when $v$ is not.

\medskip
\begin{remar} Let $\wh\mu$ be the extension of measure $\mu$ defined on an exact finite rank subdiagram $\ol B$ of a Bratteli diagram $B$. Suppose that $\widehat{\mu}(\wh X_{\ol B}) < \infty$.
Then we have
\begin{eqnarray*}
\max\limits_{v \in W_n'}\widehat{\mu} (X_v^{(n)}) &\leq &\sum_{v\in W'_n}\widehat{\mu} (X_v^{(n)})\\
 &= &\widehat{\mu}(\wh X_{\ol B}) - \sum_{w\in W_n}\widehat{\mu}(X_w^{(n)}) \rightarrow  0 \mbox{  as  } n \rightarrow \infty.
\end{eqnarray*}
Since the measure of any tower $\ol X_w^{(n)}$ is bounded from zero, it follows that
\begin{equation}
\lim_{n \rightarrow \infty}\frac{\max\limits_{v \in W_n'}\widehat{\mu} (X_v^{(n)})}{\min\limits_{w \in W_n} \mu (\ol X_w^{(n)})} = 0,
\end{equation}
and therefore
\begin{equation}\label{ratio}
\lim_{n \rightarrow \infty}\frac{\max\limits_{v \in W_n'}\widehat{\mu} (X_v^{(n)})}{\min\limits_{w \in W_n} \widehat{\mu}(X_w^{(n)})} = 0.
\end{equation}

\end{remar}

It is very plausible that (\ref{ratio}) is true for any uniquely ergodic Bratteli subdiagram $\ol B$ but this question remains open. On the other hand we are able to prove the following result.

\begin{prop}\label{PropMin2Max}
Let $B$ be a Bratteli diagram with incidence matrices $F_n = \{(f_{v,w}^{(n)})\}$. Let $\ol B = \ol B(W)$ be a proper vertex subdiagram of $B$ such that $\dfrac{|W_{n}|}{|V \setminus W_{n}|} \leq C$ for every $n$ and some constant $C>0$. Suppose $\widehat{\mu}$ is a finite invariant measure on the path space $X_B$ which is obtained as the extension of a probability measure $\mu$ defined on $X_{\overline{B}}$. Then
\begin{equation}\label{min2max}
\lim_{n \rightarrow \infty}\frac{\min\limits_{w \in W_n'}\widehat{\mu} (X_v^{(n)})}{\max\limits_{w \in W_{n}} \mu (\ol X_w^{(n)})} = 0.
\end{equation}
\end{prop}

\begin{proof}
Let $W'_n = V_n \setminus W_n$. For every $n$, we have
\begin{eqnarray*}
\widehat{\mu}(\wh X_{\ol B})  &= &\sum_{v \in W_n'}\widehat{\mu} (X_v^{(n)}) + \sum_{w \in W_n}\widehat{\mu} (X_w^{(n)})\\
& \geq &|W_n'| \min_{v \in W_n'}\widehat{\mu} (X_v^{(n)}) + \sum_{w \in W_n}\widehat{\mu} (X_w^{(n)}).
\end{eqnarray*}
Hence
$$
\min_{v \in W_n'}\widehat{\mu} (X_v^{(n)}) \leq \frac{\widehat{\mu}(X_B) - \sum_{w \in W_n}\widehat{\mu} (X_w^{(n)})}{|W_n'|}.
$$
Since $\mu(X_{\overline{B}}) = 1$, we obtain
 $$
\max\limits_{w \in W_{n}}{\mu}(\ol X_w^{(n)}) \geq \frac{1}{|W_{n}|}.
$$
Hence,
$$
\frac{\min\limits_{w \in W_n'}\widehat{\mu} (X_v^{(n)})}{\max\limits_{w \in W_{n}} {\mu} (\ol X_w^{(n)})} \leq \frac{|W_{n}|(\widehat{\mu}(X_B) - \sum_{w \in W_n}\widehat{\mu} (X_w^{(n)}))}{|W_{n}'|}.
$$
Notice that $\widehat{\mu}(X_B) - \sum_{w \in W_n}\widehat{\mu} (X_w^{(n)}) \rightarrow 0$ as $n \rightarrow \infty$. This proves that equality~(\ref{min2max}) holds.

\end{proof}

\begin{remar}
Since $\widehat{\mu} (X_w^{(n)}) \geq {\mu} (\ol X_w^{(n)})$ for every $w \in W_n$ and every $n$, we obtain the following simple corollary of the proved result
$$
\lim_{n \rightarrow \infty}\frac{\min\limits_{v \in W_n'}\widehat{\mu} (X_v^{(n)})}{\max\limits_{w \in W_{n}} \widehat{\mu} (X_w^{(n)})} = 0.
$$
\end{remar}

\textbf{Acknowledgment.} This article was finished when the first named author
was a visitor of the Department of Mathematics, University of Iowa. He is thankful to the department for the hospitality and support.

\emph{Institute for Low Temperature Physics, Kharkov, Ukraine} and \emph{Faculty of Mathematics and Computer Science of Warmia and Mazury University, Olsztyn, Poland}\\
\emph{E-mail address: bezuglyi@ilt.kharkov.ua, helen.karpel@gmail.com, jkwiat@mat.umk.pl}

\end{document}